\documentclass{article}
\usepackage{graphicx} 
\usepackage{color} 
\usepackage{amsthm} 
\usepackage{amsmath}
\usepackage{amsfonts}
\usepackage{mathabx} 
\usepackage{xcolor}
\usepackage[normalem]{ulem}
\usepackage{comment}
\usepackage{thmtools, thm-restate} 
\usepackage[utf8]{inputenc}
\usepackage{graphicx}
\usepackage{subcaption}
\usepackage{lineno}
\usepackage{diagbox} 
\usepackage{multirow}

\newtheorem{theorem}{Theorem}
\newtheorem{conj}[theorem]{Conjecture}

\newtheorem{problem}{Problem}

\newtheorem{corollary}[theorem]{Corollary}

\newtheorem{lemma}[theorem]{Lemma}

\providecommand{\keywords}[1]{
  \par
  \addvspace\baselineskip 
  \noindent\textbf{\textit{Keywords---}} #1
}

\title{The balanced upper chromatic number of linear hypergraphs and the $n$-cube over $t$ elements
\thanks{Part of this research was performed during a Simons Laufer Mathematical Sciences Institute
(SLMath, formerly MSRI) summer research program, which is supported by the National Science Foundation (Grant No. DMS-1928930) and in partnership with the Mathematics Institute of the National Autonomous University of Mexico (UNAM).}
}

\author{Gabriela Araujo-Pardo\thanks{Instituto de Matem\'aticas, Universidad Nacional Aut\'onoma de M\'exico, Campus Juriquilla, Mexico, {\tt garaujo@im.unam.mx}, supported by DGAPA PAPIIT IN113324, SECHITI: CBF2023-2024-552}
	\and
	Silvia Fern\'andez-Merchant\thanks{California State University, Northridge, 18311 Nordhoff St, Northridge, CA, 91330, USA. {\tt silvia.fernandez@csun.edu}, supported by a 2024 RSP CSUN Campus Funding Initiative}
    \and
    Adriana Hansberg\thanks{Instituto de Matem\'aticas, Universidad Nacional Aut\'onoma de M\'exico, Campus Juriquilla, Mexico, {\tt ahansberg@im.unam.mx}, supported by DGAPA PAPIIT IG100822}
    \and 
    Dolores Lara\thanks{Departamento de Computaci\'on, Centro de Investigaci\'on y Estudios de Posgrado del IPN, Mexico City, Mexico, {\tt dolores.lara@cinvestav.mx}}
    \and 
    Amanda Montejano\thanks{UMDI, Facultad de Ciencias, UNAM Juriquilla, Quer\'etaro, Mexico, {\tt amandamontejano@ciencias.unam.mx}, supported by DGAPA PAPIIT IG100822 y IN113626}
    \and 
    D\'eborah Oliveros\thanks{Instituto de Matem\'aticas, Universidad Nacional Aut\'onoma de M\'exico, Campus Juriquilla, Mexico, {\tt doliveros@im.unam.mx}, supported by DGAPA PAPIIT -IN112124, IN109826 SECHITI: CBF2023-2024-552}
 }

\begin{document}
\maketitle

\begin{abstract}
A coloring of the vertices of a hypergraph is called \emph{balanced} if the sizes of the color classes differ by at most one. We say that a hyperedge is \emph{rainbow} if its elements have pairwise distinct colors. In this paper, we provide a general upper bound on the \emph{balanced upper chromatic number} of arbitrary linear hypergraphs, that is, the largest integer $k$ such that there exists a balanced $k$-coloring of the vertices of the hypergraph without rainbow hyperedges. We focus on the cube $C_t^n$, defined as the linear hypergraph whose vertices are the lattice points in $[0,t-1]^n$, and whose hyperedges are the sets of $t$ collinear points. 
We determine the exact balanced upper chromatic number of $C_t^n$ for $t\geq 4n-2$. For smaller values of $t$, we present bounds and determine this parameter (with few exceptions) in dimensions $2$ and $3$.
\end{abstract}
\keywords{Balanced upper chromatic number, Linear hypergraph, Hypercube}

\section{Introduction}
Let $H$ be a $t$-uniform hypergraph of order $v$. $H$ is called \emph{linear} if its edges intersect pairwise in at most one vertex. Given a coloring of the vertices of $H$, we say that a hyperedge is \emph{rainbow} if its elements have pairwise distinct colors. A coloring is called \emph{balanced} if the sizes of the color classes differ by at most one. The \emph{balanced upper chromatic number} of $H$, denoted by $\overline{\chi}_b(H)$, is the largest integer $k$ such that  there is a balanced $k$-coloring of the vertices of $H$ without rainbow hyperedges. The \emph{$n$-cube on $t$ elements}, denoted by $C_t^n$, is defined as the set of lattice points with integer coordinates in the interval $[0,t-1]$. The \emph{geometric lines} of this cube are all the subsets of $t$ collinear points. They satisfy that, for each $0\leq i\leq t-1$, the coordinates are all equal to some fixed value or equal to the vector $(0,1, \ldots, t-1)$ either appearing  in increasing or decreasing order. We will regard $C_t^n$ as a hypergraph by taking the geometric lines as its hyperedges, see Section ~\ref{sec:n-cube_definition} for precise definitions and notation around the concept of the $n$-cube $C_t^n$. In this paper, we  study the balanced upper chromatic number of linear hypergraphs and, more specifically, of the $n$-cube on $t$ elements.

The study of the balanced upper chromatic number comes back to the setting of rainbow Ramsey problems. Broadly speaking, Ramsey theory usually studies the presence of monochromatic patterns in colorings of mathematical universes. Typically, one is interested in the smallest possible size of a set of elements such that, under any arbitrary coloring, it inevitably contains certain special monochromatic subsets. For example, the classical Ramsey theorem in Graph theory states that every $k$-coloring of the set of edges of a complete graph $K_n$ contains a monochromatic copy of $K_t$, provided that $n$ is sufficiently large~\cite{Ram29}. In this context, the classical Ramsey number $R(k;t)$ denotes the minimum integer such that, for any $n \ge R(k;t)$, every $k$-coloring of the edges of $K_n$ contains a monochromatic copy of $K_t$. In the arithmetic setting, the van der Waerden number $w(k;t)$ denotes the minimum integer such that, for every $n \ge w(k;t)$, every $k$-coloring of the set of the first $n$ integers contains a monochromatic $t$-term arithmetic progression \cite{vW27}. An important result that implies van der Waerden's theorem is the well-known Hales-Jewett theorem \cite{bpv09,HJ63} that establishes the existence of monochromatic \emph{combinatorial lines}. The set of combinatorial lines is a subset of the set of geometric lines as defined above, see Section~\ref{sec:n-cube_definition} or \cite{bpv09} for a precise definition. See \cite{GRSS15, S10} for more information on Ramsey Theory.

When looking at color patterns that are different from monochromatic, as for instance rainbow, lexicographic, balanced, or omnitonal patterns (see, for instance, \cite{AxFDF04, CHM21, ErRa50, jlmnr03, hm15}), one is forced to impose certain conditions on the colorings to guarantee an adequate representation of the colors. In the context of rainbow Ramsey theory, which studies the presence of rainbow color patterns, it is thus very usual that density conditions on the color classes are imposed (see \cite{AxFDF04,jlmnr03,hm15}). In this sense, an optimal situation is given when the colors are equally represented, i.e. when the coloring is balanced. Also, as having many colors increases the chance for the presence of rainbow hyperedges, another approach for guaranteeing the existence of the rainbow pattern is to consider more color classes. For instance, 
it is known that, for $n \geq 3$, a $3$-coloring of $[n]$ with at least $\frac{n}{6}$ elements in each color, a rainbow $3$-term arithmetic progression exists \cite{AxFDF04}. This result was originally conjectured in \cite{jlmnr03}, where it is also shown that, for $t \ge 4$ colors, no density condition on the color classes is sufficient to assure the presence of rainbow $t$-term arithmetic progressions. Thus, a natural question in such a situation is to ask for the maximum number of color classes in which rainbow arithmetic progressions (that can be interpreted as hyperedges) could possibly be avoided in balanced colorings, which is the very essence of the balanced upper chromatic number. Precisely for this reason, in \cite{jlmnr03}, the authors propose the study of a parameter whose definition is exactly the same as the upper chromatic number with the particularity that it is restricted to \emph{equinumerous colorings}, that is, colorings where all color classes have exactly the same size (requiring therefore also some divisibility conditions on $n$).

 The balanced upper chromatic number has been studied for cyclic projective planes, projective spaces, desarguesian projective planes and for the cube in \cite{akm, bbmns21, montejano}. In particular, upper bounds on the balanced upper chromatic number of finite projective planes, which constitute a special family of linear hypergraphs, are given in \cite{akm, bbmns21}. For the cube $C_t^n$, the nontrivial lower bound 
 \begin{equation}\label{eq:first_nt_lower}
     \overline{\chi}_b(C_t^n)\geq \left(\frac{t}{2}\right)^n
 \end{equation}
was obtained in \cite{montejano} for any even $t\geq 4$. This bound makes the dependency of $\overline{\chi}_b(C_t^n)$ on the number of points $t^n$ evident, in contrast to arithmetic progressions, where the maximum number of colors such that there is an equinumerous coloring without rainbow $t$-term arithmetic progressions for the set $[n]$ does not depend on $n$ \cite{jlmnr03}. 

Finally, we note here that there is a somehow similar parameter called the upper chromatic number, which 
was introduced in~\cite{voloshin1995upper},
that has been extensively studied for various types of hypergraphs, see \cite{ara03, bhs13, bt08, vol95}. It is defined precisely as the balanced upper chromatic number but with no restriction on the sizes of the color classes. This apparently subtle difference makes a deep impact on the results and techniques, as extremal cases for the parameter without restrictions on the coloring often involve constructions with many small color classes and one very large class of one of the colors, thus going into the opposite direction of balanced colorings. A version of this parameter for the more general case of the so-called mixed hypergraphs is treated in~\cite{montellano2008}.

Other related approaches involving arithmetic structures are handled in~\cite{hm15, act89, dlmor16, fjr07}.

\subsection{The $n$-cube and its geometric lines}\label{sec:n-cube_definition}

We consider the \textit{$n$-cube over $t$ elements}, denoted by $C^n_t$, defined as  the set of  \textit{points} (ordered $n$-tuples) on the set $\{0,1,\dots,t-1\}$. That is, $$C_t^n=\{\textbf{x}=(x_1,x_2,\dots,x_n)\,:\,0\leq x_i\leq t-1, x_i\in\mathbb{Z}\}.$$ 
We use bold fonts to represent the points \textbf{x} of $C_t^n$ as above. A \textit{geometric line} in the $n$-cube $C_t^n$ consists of exactly $t$ collinear points $\textbf{x}_0,\textbf{x}_1,\textbf{x}_2,\dots \textbf{x}_{t-1}$ of $C^n_t$. Formally, a set of $t$ distinct points of $C_t^n$ is a geometric line if there is an order of the points, such that, when we write their coordinates in the following array
    \begin{alignat*}{9}
        &\textbf{x}_0& &\,=\,& &(x_{0,1},& &x_{0,2},& &\dots \quad&  &x_{0,j},& &\dots\quad&  &x_{0,n-1},& &  x_{0,n})\\
        &\textbf{x}_1& &\,=\,& & (x_{1,1},& & x_{1,2},& &\dots&  &x_{1,j},&  &\,\dots\,&  &x_{1,n-1},&&  x_{1,n})\\
        &\textbf{x}_2& &\,=\,& &(x_{2,1},& & x_{2,2},& &\dots&  &x_{2,j},&  &\,\dots\,&  &x_{2,n-1},&&  x_{2,n})\\
       &\:\vdots& && && && && &\:\vdots& && && &\:\vdots\\
       & \textbf{x}_{t-1}& &=& &(x_{t-1,1},& &x_{t-1,2},& &\:\dots& &x_{t-1,j},&  &\:\dots& &x_{t-1,n-1},& &x_{t-1,n}),
    \end{alignat*}
each of the $n$ columns  satisfies one of the conditions below:
    
\begin{enumerate}
    \item[(a)] The entries are all equal to some fixed value $a\in \{0,1,\dots,t-1\}$;
    \item[(b)] The entries appear in increasing order ${0,1,\dots,t-1}$;
    \item[(c)] The entries appear in decreasing order $t, t-1, \ldots 1,0$;
\end{enumerate}
where not all entries satisfy (a), as in this case all points $\textbf{x}_0,\textbf{x}_1,\dots,\textbf{x}_{t-1}$ would concide.
The lines satisfying just conditions (a) or (b) are called \emph{combinatorial lines}. We will identify each line in $C^n_t$  with a vector $\langle l_1,l_2,\dots,l_n \rangle$ such that, for every $1\leq i \leq n$,  
\[
l_i = \left\{
\begin{array}{ll}
     a \in \{0,1, \ldots, t-1\}, &  \mbox{ if column $i$ is of type (a) with fixed value $a$;}\\
     b, &  \mbox{ if column $i$ is of type (b);}\\
     c, & \mbox{ if column $i$ is of type (c),}
\end{array}
\right.
\]

where there is at least one $l_i\in\{b,c\}$. For example, the line identified with vector $\langle 21,b,37,6,b,c,6,34,2,10 \rangle$ corresponds to the following line in $C^{10}_{41}$.
        
\begin{alignat*}{12}\label{example1}
    &\textbf{x}_0&      &\;=& &\;(21,& &\ 0,& &\ 37,&  &\ 6,& &\ 0,& &\;40,& &\ 6,& &\ 34,& &\ 2,& &\ 10) \\
    &\textbf{x}_1&      &\;=& &\;(21,& &\ 1,& &\ 37,&  &\ 6,& &\ 1,& &\;39,& &\ 6,& &\ 34,& &\ 2,& &\ 10) \\
    & \textbf{x}_2&     &\;=& &\;(21,& &\ 2,& &\ 37,&  &\ 6,& &\ 2,& &\;38,& &\ 6,& &\ 34,& &\ 2,& &\ 10) \\
   &\:\vdots& && && && && && &\ \vdots& && && && && &\quad \vdots\\
    &\textbf{x}_{38}&   &\;=&  &\;(21,& &\;38,& &\ 37,& &\ 6,& &\;38,& &\ 2,& &\; 6,& &\ 34,& &\ 2,& &\ 10) \\
    &\textbf{x}_{39}&   &\;=& &\;(21,& &\;39,& &\ 37,& &\ 6,& &\;39,& &\ 1,& &\; 6,& &\ 34,& &\ 2,& &\ 10) \\ 
    &\textbf{x}_{40}&   &\;=& &\;(21,& &\;40,& &\ 37,& &\ 6,& &\;40,& &\ 0,& &\; 6,& &\ 34,& &\ 2,& &\ 10). \\ 
\end{alignat*}

It is important to observe that there must be at least one column of type (b) or (c) so that the $t$ points are different. By convention, to avoid describing the same line in two different ways (as the points of a line can be ordered backwards, too), we always assume that the first column of type (b) or (c) must be of type (b). The set of geometric lines of $C_t^n$ is denoted by $\mathcal{L}(C_t^n)$.  Since each geometric line is contained in a straight line in $\mathbb{R}^n$, two distinct points in $C_t^n$ belong to at most one geometric line, and two different geometric lines intersect in at most one point. The number of geometric lines is known to be 
\begin{equation}\label{eq:lines}
    |\mathcal{L}(C_t^n)|=\frac{(t+2)^n-t^n}{2},
\end{equation}
as shown in \cite{bpv09}. Note that this fact can be easily seen using our description of the geometric lines above. Indeed, there are $(t+2)^n$ vectors $\langle l_1,l_2,\dots,l_n \rangle$ with entries in $\{b,c,0,1,2,\dots,t-1\}$, $t^n$ of them having all $n$ entries in $\{0,1,2,\dots,t-1\}$, which is not allowed; finally we divide by $2$ to avoid repetitions, thus counting just 
those whose first entry in $\{b,c\}$ is $b$.  

In the rest of the paper, we associate the $n$-cube over $t$ elements to the hypergraph whose set of vertices is $C_t^n$ and set of hyperedges is $\mathcal{L}(C_t^n)$. In order to simplify the exposition, we abuse the notation referring to this $t$-uniform hypergraph simply as $C_t^n$ and to its balanced upper chromatic number as $\overline{\chi}_b(C_t^n)$. 

\subsection{Outline of the paper and results}\label{sec:problem}

Let $H$ be a $t$-uniform linear hypergraph on $v$ vertices and with $e$ hyperedges.  We want to study what is the maximum number of colors one may use to color its vertices such that rainbow edges are avoided. If there are fewer than $t$ colors provided, it is easy to see that no rainbow edge exists. On the other hand, if one has as many colors as vertices, every hyperedge is rainbow. Thus we obtain easily $t-1 \le \overline{\chi}_b(H) \le v-1$. In this work, we deep further into understanding the balanced upper chromatic number for linear hypergraphs. Our results consists of a general non-trivial upper bound on $\overline{\chi}_b(H)$ for any $t$-uniform linear hypergraph $H$, as well as a sharper analysis of the special case of the $n$-cube, where we determine the exact balanced upper chromatic number of $C_t^n$ for $t\geq 4n-2$, we explore bounds for this invariant for smaller values of $t$, and (except for a few cases) determine its behavior in dimensions $2$ and $3$.

In Section~\ref{sec:general_upper}, we present a general upper bound for the balanced upper chromatic number of linear hypergraphs. Observe that, to avoid rainbow lines, we need to have at least two points of the same color in each line. In this case, we say that the color \textit{blocks} the line. As $H$ is linear, any coloring with at most $e-1$ color classes of size two and all other classes (if any) of size one would fail to block all the lines. Provided that $v\geq 2e$, the smallest of such colorings has $v - e+1$ colors and thus
\begin{equation}\label{eq:easyupper}
\overline{\chi}_b(H) \le v - e.
\end{equation}
When $v < 2e$, a more involved analysis completes our first main theorem.

\begin{restatable}{theorem}{uppergeneral}
\label{thm:uppergeneral}
    Let $H$ be a linear hypergraph with $v$ vertices and $e$ hyperedges. Then 
    \begin{equation*}
        \overline{\chi}_b(H)\leq 
        \begin{cases}
            v-e & \text{if } v\geq 2e,\\
            \left\lceil 2(v+2e)-4\sqrt{e^2+ev}\right\rceil-1 & \text{if } v< 2e.
        \end{cases}
    \end{equation*}
\end{restatable}

This bound comes very close to the known upper bounds for the upper balanced chromatic number of finite projective planes, which constitute a special family of linear hypergraphs, which were given in \cite{akm, bbmns21}. When considering the $n$-cube $C_t^n$, in the case the number of vertices $t^n$ is at least as large as twice the number of lines, bound (\ref{eq:easyupper}) is translated by means of (\ref{eq:lines}) to
\begin{equation*}
     \overline{\chi}_b(C_t^n)\leq
         \dfrac{3t^n-(t+2)^n}{2}.
\end{equation*}

The condition on the number of lines and vertices corresponds to the hypothesis that $t\geq 2/(\sqrt[n]{2}-1)$ in Corollary \ref{cor:upper}. Surprisingly, this upper bound is best possible for $t$ large enough, as shown in Section~\ref{sec:mainresult} by a quite involved construction of a balanced coloring achieving this bound that uses Hall's Marriage Theorem. More precisely, our second main theorem states:

\begin{restatable}{theorem}{mainresult}
\label{thm:mainresult}
 For integers $n\geq 2$ and $t\geq 4n-2$, 
\begin{equation*}
    \overline{\chi}_b(C_t^n) =\frac{3t^n-(t+2)^n}{2}.
\end{equation*}
\end{restatable}

In Section~\ref{sec:small_t}, we provide a lower bound for $\overline{\chi}_b(C_t^n)$ for even $t<4n-2$, see Theorem~\ref{th:lowernew2}. Finally, in Section \ref{sec:conclusions}, we present a summary of our results (see the table in Figure \ref{fig:summary}), which includes the exact balanced upper chromatic (with few exceptions) in dimensions $2$ and $3$; as well as a conjecture and several open problems for future work. Preliminary versions of this paper appeared in \cite{DMD24,CCCG24}.

\section{General upper bound for linear hypergraphs}\label{sec:general_upper} 

We start by presenting an upper bound for the balanced upper chromatic number of any linear hypergraph.

\uppergeneral*

\begin{proof}
    The case when $v\geq 2e$ was already settled in the Introduction in Expression~(\ref{eq:easyupper}). Thus, we may assume that $v< 2e$.
    Consider a balanced $c$-coloring of $H$ for some integer $2\leq c \leq v$. Then there is an integer $1\leq k <  v$ such that all color classes are of size $k$ or possibly $k+1$. Let $c_k\geq1$ and $c_{k+1}\geq 0$ be the number of classes of size $k$ and $k+1$, respectively. Then $c=c_k+c_{k+1}$ and $v=kc_k+(k+1)c_{k+1}=ck+c_{k+1}$. Since $0\leq c_{k+1}<c$, then $k$ and $c_{k+1}$ are the quotient and the remainder, respectively, when $v$ is divided by $c$. That is, $k=\lfloor \frac{v}{c} \rfloor$ and $c_{k+1}=v-c\lfloor \frac{v}{c} \rfloor$. Let $\varepsilon=\frac{c_{k+1}}{c}=\frac{v}{c} -\lfloor \frac{v}{c} \rfloor$. Thus $c_{k+1}=\varepsilon c$ and $k=\frac{v}{c}-\varepsilon$, where $0\leq \varepsilon<1$.    We say that a color \textit{blocks} a hyperedge if at least two vertices of the hyperedge receive that color. So an \textit{unblocked} edge is a rainbow edge. Note that at most $\binom{k}{2}c_k+\binom{k+1}{2}c_{k+1}=\binom{k}{2}c+kc_{k+1}$ hyperedges can be blocked by the distinct colors in the $c$-coloring. Hence, if 

    \begin{align}\label{eq:1rainbow}
        e >\binom{k}{2}c+kc_{k+1} 
        &=\binom{\frac{v}{c}-\varepsilon}{2}c+\left(\frac{v}{c}-\varepsilon\right) \varepsilon c \nonumber \\
        &=\frac{1}{2}\left(\frac{v}{c}-\varepsilon\right)(v+\varepsilon c-c) \nonumber \\
        &=\frac{1}{2c}\left(v-\varepsilon c\right)(v+\varepsilon c-c),
    \end{align}
then there is at least one rainbow hyperedge. 

     Note that (rationalizing)
     \begin{multline}\label{eq:rationalization}
         c>2(v+2e)-4\sqrt{e^2+ev}=\frac{2v^2}{v+2e+2\sqrt{(e^2+ev)}} \\
         =  \frac{2v^2}{v+2e+\sqrt{(v+2e)^2-v^2}}.
     \end{multline}
    Because $(v+2e)^2-v^2< (v+2e)^2$, it follows that $ c>\frac{v^2}{v+2e}$.
    This implies that $e>\frac{v(v-c)}{2c}$, which is precisely Inequality (\ref{eq:1rainbow}) when $\varepsilon=0$. 
    
    Assume now that $0<\varepsilon<1$. Thus $0 < \varepsilon(1-\varepsilon)\leq \frac{1}{4}$. Since $2e>v$ and $v\geq c$, we have that
    \begin{equation}\label{eq:positive}
        \frac{v+2e}{2\varepsilon(1-\varepsilon)}> \frac{v}{\varepsilon(1-\varepsilon)}\geq 4v>v\geq c.
    \end{equation}
    Also, $v^2\geq 4v^2\varepsilon(1-\varepsilon)$. By Inequality (\ref{eq:rationalization}) and rationalizing, we obtain
    \begin{align}\label{eq:otherside}
        c >\frac{2v^2}{v+2e+\sqrt{(v+2e)^2-4v^2\varepsilon(1-\varepsilon)}} \nonumber \\
        &=\frac{v+2e-\sqrt{(v+2e)^2-4v^2\varepsilon(1-\varepsilon)}}{2\varepsilon(1-\varepsilon)}.
    \end{align}
    Inequalities (\ref{eq:positive}) and (\ref{eq:otherside}) imply that
    \begin{equation*}
        \frac{\sqrt{(v+2e)^2-4v^2\varepsilon(1-\varepsilon)}}{2\varepsilon(1-\varepsilon)}>
        \frac{v+2e}{2\varepsilon(1-\varepsilon)}-c>0.
    \end{equation*}
    Multiplying by $2\varepsilon(1-\varepsilon)$ and squaring, we obtain
  
    \begin{align*}
        (v+2e)^2-4v^2\varepsilon(1-\varepsilon)&>(v+2e-2\varepsilon(1-\varepsilon)c)^2\\
        -4v^2\varepsilon(1-\varepsilon)&>-4(v+2e)\varepsilon(1-\varepsilon)c+4(\varepsilon(1-\varepsilon)c)^2\\
         -v^2&>-(v+2e)c+\varepsilon(1-\varepsilon)c^2\\
         e&>\frac{1}{2c}\left(v^2-vc+\varepsilon c^2 -\varepsilon^2 c^2\right) \\
         &=\frac{1}{2c}\left(v-\varepsilon c\right)(v+\varepsilon c-c),
    \end{align*}
    This is again Inequality (\ref{eq:1rainbow}) and thus there must be a rainbow hyperedge.
\end{proof}

As was mentioned in the Introduction, the special case of the linear hypergraphs of finite projective planes were studied in \cite{akm, bbmns21}. Since the finite projective planes $\Pi_q$ of order $q$ have the same number of lines (of size $q+1$) and points, namely $v = e = q^2 + q + 1$, the upper bound above, which in this setting falls into the second case, gives $\overline{\chi}_b(\Pi_q)\leq (6-4\sqrt{2}) v \approx 0.34 v$, which falls very close to the previous bound of $\overline{\chi}_b(\Pi_q)\leq v/3$ \cite{akm, bbmns21}.

The following result is a direct application of Theorem~\ref{thm:uppergeneral} with $v=|C_t^n|=t^n$ and $e=|\mathcal{L}(C_t^n)|=((t+2)^n-t^n)/2$.

\begin{corollary}\label{cor:upper} 
Let $t$ and $n$ be positive integers. Then
\begin{equation*}
     \overline{\chi}_b(C_t^n)\leq
     \begin{cases}
         \dfrac{3t^n-(t+2)^n}{2} & \text{if\hspace{.1in}} t\geq\dfrac{2}{\sqrt[n]{2}-1},\\
         \left\lceil 2(t+2)^n-2\sqrt{(t+2)^{2n}-t^{2n}}\right\rceil-1 & \text{if\hspace{.1in}} 2\leq t < \dfrac{2}{\sqrt[n]{2}-1}. 
     \end{cases}
\end{equation*}
\end{corollary}

It can be checked that when $t=2$, this result implies $\overline{\chi}_b(C_2^n)=1$, as mentioned before.

\section{The exact value of $\overline{\chi}_b(C_t^n)$ for large $t$}\label{sec:mainresult}

In order to prove Theorem~\ref{thm:mainresult}, we provide an explicit balanced $(\frac{3t^n-(t+2)^n}{2})$-coloring of $C_t^n$ with no rainbow lines, for every $n\geq 2$ and $t\geq 4n-2$, attaining thus the upper bound in Corollary \ref{cor:upper}. This is equivalent to proving the existence of what we call a \emph{double-matching} of $C_t^n$, which could be intuitively described as a ``disjoint'' selection of two points per line. Hence, one could also refer to a double-matching as a \emph{double-covering}, a \emph{double-transversal}, or a \emph{double-SDR} (a ``double'' System of Distinct Representatives).
To complete this task, we use an auxiliary injective function defined in Section~\ref{sec:injective_function} and heavily make use of the symmetry of the lines within the $n$-cube as described in Section~\ref{sec:equivalence}. The double-matching is provided in Section~\ref{sec:matching}, which includes the core of the proof.

\subsection{An injective function}\label{sec:injective_function}

Let $m$ and $k$ be integers with $0\leq k\leq \frac{1}{2}(m-1)$. The Hall's Marriage Theorem guarantees the existence of an injective function $$g_{m,k}:\binom{\{0,1,2,\dots,m-1\}}{k} \rightarrow \binom{\{0,1,2,\dots,m-1\}}{k+1}$$ such that, for any $S\subset \{0,1,2,\dots,m-1\}$ with $k$ elements, it holds that $S\subset g_{m,k}(S)$. An explicit such function $g_{m,k}$ was given in \cite{A73}. This function, adapted to our setting, can be defined as follows. Given $S=\{s_1,s_2,\dots, s_k\}\subset \{0,1,2, \dots, m-1\}$ where $s_1<s_2<\dots< s_k$, let $s_0=-1$, $r=\min \{s_i-2i:i\in \{0,1,2,\dots,k\}\}$, and $\phi(S)=\max\{i\in \{0,1,2,\dots,k\}:s_i-2i=r\}$, the largest of the subindices for which $s_i-2i$ is as small as possible. 
Then $g_{m,k}(S)=S\cup \{1+s_{\phi(S)}\}$. Note that when $k=0$, we have that $S=\emptyset$, $\phi(S)=0$, and $g_{m,k}(S)=\{0\}$. 

\subsection{Symmetric pairs and equivalence classes}\label{sec:equivalence}

Let $t\geq 2$ and $0\leq m \leq t-1$ be integers. Let $\overline{m}=\overline{t-1-m}=\{m,t-1-m\}$. We say that $m$ and $t-1-m$ are \emph{symmetric} and refer to $\overline{m}$ as a \emph{symmetric pair}. Further, we write $\widehat{m}$.
to refer to the smallest element in the pair $\overline{m}$. It is important to note that when $t$ is odd and $m=(t-1)/2$, we have that $m=t-1-m$ and so $\overline{m}$ consists of a single element. For convenience, we still call $\overline{(t-1)/2}$ a symmetric \emph{pair} observing that it is actually not used later in the proof when an actual pair of points is selected.

We partition the set of points $C_t^n$ into equivalence classes. We say that two points $(x_1,x_2,\dots,x_n)$ and $(y_1,y_2,\dots,y_n)$ are equivalent if $x_i=y_i$ or $x_i+y_i=t-1$ for all $1\leq i\leq n$, that is, if each pair of corresponding entries is a symmetric pair. The equivalence class containing the point ${\bf x}=(x_1,x_2,\dots,x_n)$ is denoted by $\overline{{\bf x}}=\{(y_1,y_2,\dots,y_n):y_i\in\overline{x_i} \text{ for all }1\leq i \leq n\}$. 

Note that two points in the same line are equivalent if and only if they are symmetric around the center of the line. In other words, a line and an equivalence class either intersect in exactly two points or do not intersect at all. For simplicity, if the elements of $A$ are listed, we avoid the braces when applying any functions to $A$. For example, we typically write  $\phi(3,4)$ instead of $\phi(\{3,4\})$.

\subsection{A double-matching}\label{sec:matching}
Let $t\geq 4n-2$ be an integer. In what follows, we associate each line in $\mathcal{L}(C_t^n)$ with two of its points in such a way that each point is associated to at most one line. We refer to this association as a \emph{double-matching} for $C_t^n$. Moreover, our double-matching satisfies that the two points associated to each line are equivalent. We present this matching as a function $f:\mathcal{L}\left(C_t^n\right)\rightarrow \binom{\, C_t^n}{2} \,$, where $f(L)\subset L$ and $f(L)\cap f(L')=\emptyset$ for any two distinct lines $L,L'\in \mathcal{L}(C_t^n)$.  

For a given line $L\in \mathcal{L}\left(C_t^n\right)$ with vector $\langle l_1,l_2,\dots,l_n \rangle$, let $A_L = \{ l_i :  l_i \in\{0,1,2,\dots,t-1\}\}$ and $\widehat{A}_L=\{\widehat{l_i}:l_i\in A_L\}$. Then $\widehat{A}_L\subset \{0,1,\dots,\lceil\frac{t}{2}\rceil-1\}$, and, setting $k =|\widehat{A}_L|$, we have $0\leq k \leq n-1\leq  \frac{1}{2}(\frac{t}{2}-1) \leq \frac{1}{2}(\lceil\frac{t}{2}\rceil-1)$. Observe that selecting a point ${\bf x}_w \in L$ means fixing an index $w\in [0,t-1]$  such that the $i^{th}$ entry of $x_w$ satisfies that
\begin{equation*}
    x_{w,i}=
    \begin{cases}
        l_i & \text{if } l_i \in A_L,\\
        w & \text{if } l_i =b,\\
        t-1-w & \text{if } l_i =c.
    \end{cases}
\end{equation*} 
The choice of $w$ is made by means of the extra element assigned to the set $\widehat{A}_L=\{s_1,s_2,\dots,s_k\}$ with $s_1< s_2<\dots< s_k$ by the injective function $g_{\lceil\frac{t}{2}\rceil,k}$. More precisely,  we define

\begin{equation}\label{eq:inclusion}
    f(L)= L \cap \overline{\bf x},
    \end{equation}
where 
\begin{equation*}
    x_i=
    \begin{cases}
        \widehat{l_i} & \text{ if }l_i\in A_L,\\
       1+s_{\phi(\widehat{A}_L)} & \text{ if }l_i\notin A_L.
    \end{cases}
\end{equation*}

Then the set of different coordinates of ${\bf x}$ is equal to $\widehat{A}_L\cup\{1+s_{\phi(\widehat{A}_L)}\}=g_{\frac{t}{2},k}(\widehat{A}_L)$. This means, in particular, that $1+s_{\phi(\widehat{A}_L)} \in \{ 0,1,\dots,\lceil\frac{t}{2}\rceil-1\} \setminus \widehat{A}_L$, and that $x_i \in \{0,1,\dots,\lceil\frac{t}{2} \rceil-1\}$ for all $1 \le i \le n$. 
Hence, if the point ${\bf y}$ is one of the two points matched to $L$, then its entries have the form
\begin{equation}\label{eq:point1}
\begin{tabular}{lrr}
      $y_i=
    \begin{cases}
        l_i & \text{ if }l_i\in A_L,\\
        1+s_{\phi(s_L)} & \text{ if }l_i=b,\\
        t-1-\left(1+s_{\phi(\widehat{A}_L)}\right) & \text{ if }l_i=c,
    \end{cases}$
    & \hspace{.1in } 
    \end{tabular}  \\
\end{equation}
for all $1\leq i\leq n$,  or
\begin{equation}\label{eq:point2}
 \begin{tabular}{lr}
    $y_i=
    \begin{cases}
        l_i & \text{ if }l_i\in A_L,\\
        t-1-\left(1+s_{\phi(\widehat{A}_L)}\right) & \text{ if }l_i=b,\\
       1+s_{\phi(\widehat{A}_L)} & \text{ if }l_i=c,
    \end{cases}$
    & \hspace{.1in } 
 \end{tabular} 
\end{equation}
for all $1\leq i\leq n$.

\noindent {\bf{Example:}} \\
Following up the geometric line $L\in \mathcal{L}(C^{10}_{41})$  defined in Section~\ref{sec:n-cube_definition},  with vector 
$\langle 21,b,37,6,b,c,6,34,2,10 \rangle$, we have that $A_L=\{21,37,6,34,2,10\}$ and  $\widehat{A}_L=\{\widehat{21},\widehat{37},\widehat{6},\widehat{34},\widehat{2},\widehat{10}\}=\{2,3,6,10,19\}$ because 
$\widehat{21}=19, \widehat{37}=3, \widehat{6}=\widehat{34}=6, \widehat{2}=2$ and $\widehat{10}=10$. Applying the function $\phi$ in Section~\ref{sec:injective_function} to $S=\widehat{A}_L$, we have $s_0=-1,s_1=2, s_2=3, s_3=6, s_4=10, s_5=19$, $r=\min\{s_i-2i:0\leq i\leq 5\}=\min\{-1,0,-1,0,2,9\}=-1$, and the maximum of the subindices achieving $r$ is  $\phi(S)=\phi(\widehat{A}_L)=\max\{0,2\}=2$. 

This yields 
\begin{align*}
    {\bf{x}} &= (19,1+s_2,3,6,1+s_2,1+s_2,6,6,2,10) \\\ &=
(19,4,3,6,4,4,6,6,2,10),
\end{align*}
and 

\begin{align*}
    f(L)&= L \cap \overline{\bf x} \\\ &=\{(21,4,37,6,4,36,6,34,2,10), (21,36,37,6,36,4,6,34,2,10)\} \\\
    &=\{  \textbf{x}_4,   \textbf{x}_{36}\}.
\end{align*}

\begin{theorem}\label{thm:dm}
    The function $f$ is a double-matching of $C_t^n$ for any integers $n\geq 2$ and $t\geq 4n-2$.
\end{theorem}
\begin{proof}
Using the same notation as the one used above for the definition of the function $f$, we first argue that $f$ is well-defined. Although, it is clear (by definition) that the points given by (\ref{eq:point1}) and (\ref{eq:point2}) are on the line $L$, we need to verify that these two points are actually different. This is clear when $1+s_{\phi(\widehat{A}_L)}\neq t-1-(1+s_{\phi(\widehat{A}_L)})$. But these two values could potentially be equal when $t$ is odd and $1+s_{\phi(\widehat{A}_L)}=\frac{t-1}{2}$. We note that this cannot happen due to how $s_{\phi(\widehat{A}_L)}$ is chosen. Indeed, since $\phi(\widehat{A}_L)\leq k\leq n-1$ and $s_{\phi(\widehat{A}_L)}-2\phi(\widehat{A}_L)\leq s_i-2i$ for all $0\leq i \leq k$, then (using $i=0$) we have
\[s_{\phi(\widehat{A}_L)} \le 2\phi(\widehat{A}_L) - 1 \le 2k-1 \le 2(n-1) -1 = 2n-3.\]
Moreover, the condition $t\geq 4n-2$ for $t$ odd is equivalent to $t\geq 4n-1$ and thus 
\[1+s_{\phi(\widehat{A}_L)} \le 2n-2 \le 2 \left( \frac{t+1}{4}\right) - 2 = \frac{t-3}{2}<\frac{t-1}{2}.\]

We now check the two conditions $f(L)\subset L$ and $f(L)\cap f(L')=\emptyset$ for any two distinct lines $L,L'\in \mathcal{L}(C_t^n)$. The first condition is guaranteed by (\ref{eq:inclusion}). To prove the second condition, suppose that ${\bf y} \in f(L)\cap f(L')$ for some point ${\bf y}=(y_1,y_2,\dots,y_n)\in C_t^n$ and for some lines $L,L'\in \mathcal{L}(C_t^n)$ identified with the vectors $\langle l_1,l_2,\dots,l_n \rangle$ and $\langle l'_1,l'_2,\dots,l'_n \rangle$, respectively. We prove that $L=L'$. Consider the set $B = \{\widehat{y_i} : 1 \leq i \leq n\}$ and say $|B|=k+1$. By definition of $f$, we have that $B = g_{\lceil\frac{t}{2}\rceil,k}(\widehat{A}_L)=g_{\lceil\frac{t}{2}\rceil,k}(\widehat{A}_{L'})$. Since $g_{\lceil\frac{t}{2}\rceil,k}$ is injective, then $\widehat{A}_L=\widehat{A}_{L'}$. By definition of $g$, we have that $B=\widehat{A}_L \cup \{1+s_{\phi(\widehat{A}_L)}\}=\widehat{A}_{L'} \cup \{1+s_{\phi(\widehat{A}_{L'})}\}$ and thus $
1+s_{\phi(\widehat{A}_L)}=1+s_{\phi(\widehat{A}_{L'})}$. Let $
w=1+s_{\phi(\widehat{A}_L)}=1+s_{\phi(\widehat{A}_{L'})}$ and $j$ be the smallest index $i$ such that $\widehat{y_i}=w$. Note that whenever $y_i\in \{w,t-1-w\}$, we have that $l_i,l'_i\in\{b,c\}$; and when $y_i\notin \{w,t-1-w\}$, we have that $l_i,l'_i\in\{0,1,2,\dots,t-1\}$.  More precisely,
\begin{equation*}
    l_i=l'_i=
    \begin{cases}
        y_i & \text{ if } \widehat{y_i}\in B\setminus \{w\},\\
        b &\text{ if } y_i=y_j,\\
        c &\text{ if } y_i=t-1-y_j.
    \end{cases}
\end{equation*}
Therefore, $L=L'$ concluding the proof. \end{proof}

Figure \ref{fig:dim3symm} shows a visual example of this matching when $n=3$ and $t=12$. In this figure, the pairs of red points on the same horizontal line are assigned to that line; the pairs of green points on the same vertical line are assigned to that line; and in general, the pairs of points with the same color $c$ on a geometric line $L$ and in the same direction of the a line with color $c$ indicated by the key at the bottom of the figure are assigned to $L$.

\begin{figure}[htb]
    \centering
    \includegraphics[scale=.6]{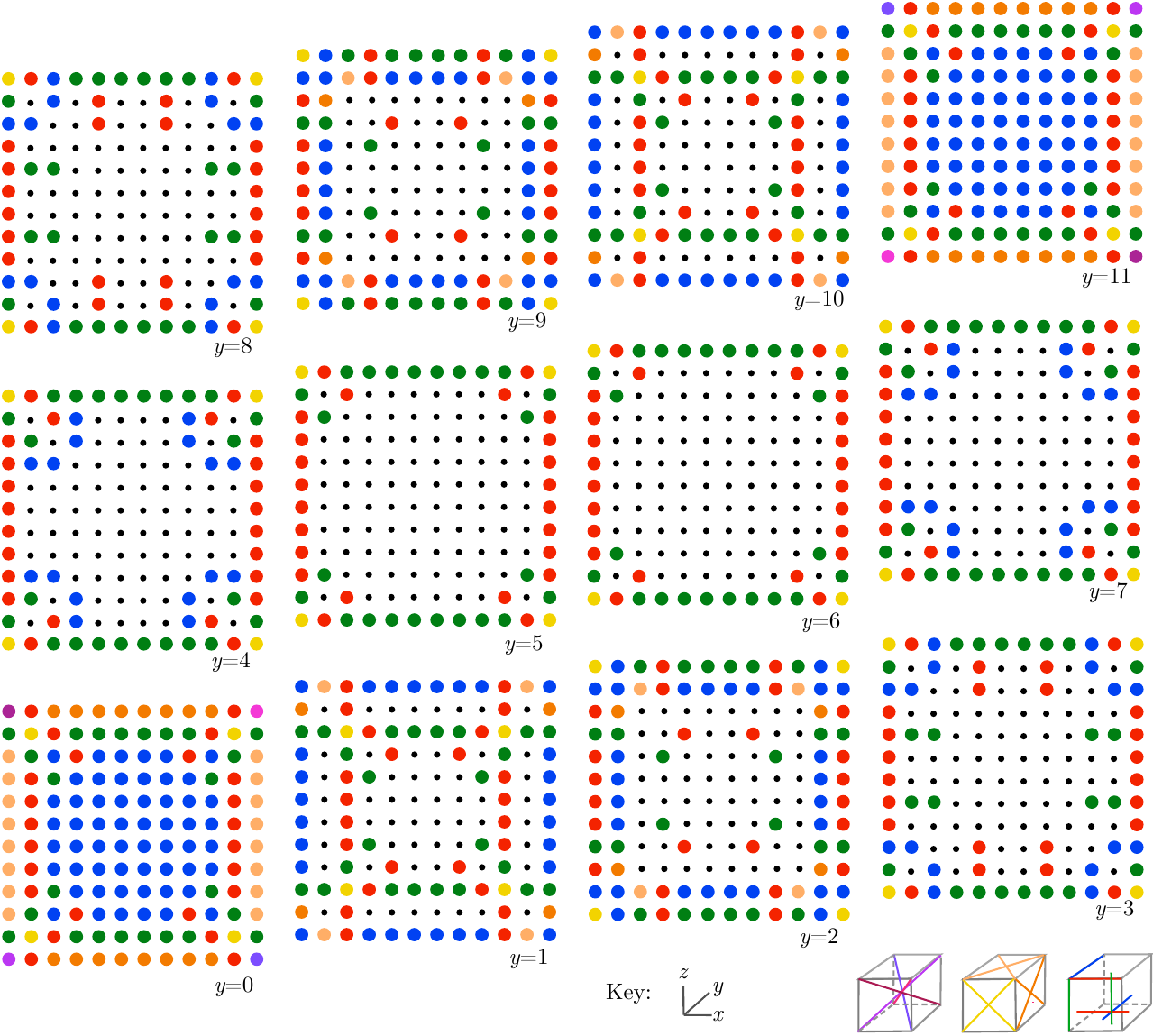}
    \caption{A double-matching for $C_{12}^3$.
    }
    \label{fig:dim3symm}
\end{figure}

\subsection{Proof of Theorem \ref{thm:mainresult}}
\label{sec:matching_to_coloring}

We are now in a position to prove our second main result, Theorem~\ref{thm:mainresult}, which provides the exact value of the balanced upper chromatic number of $C_t^n$ for $n \geq 2$ and $t \geq 4n - 2$. The proof relies on the existence of the double-matching constructed in Section~\ref{sec:matching}.

\mainresult*

\begin{proof}
Note that $t\geq 4n-2> 2/(\sqrt[n]{2}-1)$ for any $n\geq 2$ \cite{CCCG24}. The upper bound, $\overline{\chi}_b(C_t^n) \leq (3t^n-(t+2)^n)/2$, follows directly from Corollary \ref{cor:upper}. To prove the lower bound, we need to show that there is a balanced $(3t^n-(t+2)^n)/2$-coloring of $C_t^n$ with no rainbow lines. 
By Identity (\ref{eq:lines}) and since $t>2/(\sqrt[n]{2}-1)$, we have
\begin{align*}
    \frac{3t^n-(t+2)^n}{2} 
    &=t^n-\frac{(t+2)^n-t^n}{2} \\
    &=|C_t^n|-|\mathcal{L}(C_t^n)| \\
    &<|C_t^n| 
    =t^n< 3t^n-(t+2)^n.
\end{align*}

That is, the number of points in $C_t^n$ is strictly between the number of colors and twice the number of colors. Hence, in a balanced $(3t^n-(t+2)^n)/2$-coloring of $C_t^n$ each color appears once or twice. More precisely,
$|\mathcal{L}(C_t^n)|$ colors appear twice and the remaining $|C_t^n|-2|\mathcal{L}(C_t^n)|$ colors appear once. To avoid rainbow lines, each of the $|\mathcal{L}(C_t^n)|$ colors that appears twice must block a geometric line. To achieve this, we use the same color for the two points assigned to each line in $\mathcal{L}(C_t^n)$ by the function $f$, using different colors for different lines. The remaining $|C_t^n|-2|\mathcal{L}(C_t^n)|$ points use the $|C_t^n|-2|\mathcal{L}(C_t^n)|$ colors that appear once.
\end{proof}

\section{Bounds on $\overline{\chi}_b(C_t^n)$ for small $t$}\label{sec:small_t}

As we just showed, Theorem~\ref{thm:mainresult} gives exact values of the balanced upper chromatic number of the $n$-cube $\overline{\chi}_b(C_t^n)$ for $t\geq 4n-2$. In this section, we focus our attention on giving specific lower bounds for smaller values of $t$, namely, $2<t<2/(\sqrt[n]{2}-1)$. In Lemma \ref{thm:lowernew}, we present a \textit{recursive} lower bound, which, together with the lower bound for $\overline{\chi}_b(C_4^n)$ in Lemma \ref{prop:t=4}, improves the lower bound in Inequality \ref{eq:first_nt_lower} (see \cite{montejano}) for even $t$ as stated in Theorem \ref{th:lowernew2}.

\begin{lemma}\label{thm:lowernew}
Let $t$ and $n$ be positive integers and suppose that $t$ has a proper divisor $d$, $1<d<t$. Then $\overline{\chi}_b(C_t^n)\geq \left(\frac{t}{d}\right)^n\overline{\chi}_b(C_d^n)$.
\end{lemma}
\begin{proof} 
   Since $d$ divides $t$, it is possible to partition $C_t^n$ into $(t/d)^n$ congruent $n$-cubes over $d$ elements. Color each of these $(t/d)^n$ smaller $n$-cubes with $\overline{\chi}_b(C_d^n)$ different colors. This is a total of $(t/d)^n\,\overline{\chi}_b(C_d^n)$ colors. To prove that this coloring has no rainbow lines, we argue that any geometric line of $C_t^n$ completely contains a geometric line of one of the $\left(\frac{t}{d}\right)^n$ copies of $C_d^n$. Since none of these smaller lines is rainbow (i.e., each of them has at least two points of the same color), the larger line is not rainbow.
\end{proof}

\begin{lemma}\label{prop:t=4}
    For $n\geq 2$, $ \overline{\chi}_b(C_4^n)\geq 2^n+1.$
\end{lemma}
\begin{proof} 
   We begin by partitioning the cube $C_4^n$ into $2^n$ congruent $n$-cubes over $2$ elements $C_1,C_2, \dots C_{2^n}$
    and denote by $C_0$ the centered $n$-cube over $2$ elements, that is $C_0=\{\textbf{x}=(x_1,x_2,\dots,x_n)\in C_4^n:\, x_i\in \{1,2\}\}$. Consider the sets $R_i=C_i-C_0$. 
    Assign color $0$ to all vertices in $C_0$ and color $i$ to every vertex in $R_i$, $1\leq i\leq 2^n$. Note that this is a balanced partition of $C_4^n$ into $2^n+1$ parts, $2^n$ of size $2^n-1$ and one of size $2^n$, (see Figure \ref{fig:cubes_center} for this coloring of $C_4^n$ for $n=3$). To show that there are no rainbow lines, note that the first two points of any line always belong to the same cube $C_i$, for some $1\leq i \leq 2^n$. If both points are in $R_i$, then they are the same color. Otherwise, the second and third points are in $C_0$ and thus they are the same color.   
\end{proof}
\begin{figure}[htb]
    \centering
\includegraphics{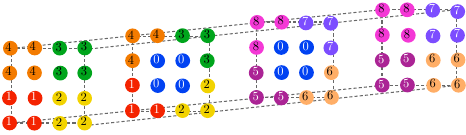}
    \caption{An illustration of the coloring in Lemma \ref{prop:t=4} for $n=3$.}
    \label{fig:cubes_center}
\end{figure}

A direct application of Lemma~\ref{thm:lowernew} and Lemma \ref{prop:t=4} gives a lower bound that improves when $t$ is a multiple of $4$. Moreover, we adapt the construction for every even $t$.

\begin{theorem}\label{th:lowernew2}
    Let $n\geq 2$. If $2\leq t\leq n$, and $t$ is even, then $\overline{\chi}_b(C_t^n)\geq \left(\frac{t}{2}\right)^n+\left\lfloor\frac{t}{4}\right\rfloor^n.$
\end{theorem}
\begin{proof}
   If $t\equiv 0 \pmod 4$), the result follows directly from Lemma~\ref{thm:lowernew} and Lemma~\ref{prop:t=4},
   namely, $\overline{\chi}_b(C_t^n)\geq \left(\frac{t}{4}\right)^n\overline{\chi}_b(C_4^n)\geq \left(\frac{t}{4}\right)^n(2^n+1)=\left(\frac{t}{2}\right)^n+\left\lfloor\frac{t}{4}\right\rfloor^n$.
   
   If $t\equiv 2 \pmod 4$, we use the construction for $t-2\equiv 0 \pmod 4$ adapted as follows (see Figure \ref{fig:t6_n3_28} for an example when $t=6$ and $n=3$). Let $S\subset C_t^n$ be the set of $(t-2)^n$ points none of whose coordinates is equal to $\frac{t}{2}-1$ or $\frac{t}{2}$. For each $\textbf{i}=(i_1,i_2,\dots,i_n)\in C_{t/2}^n$, let
   \begin{align}\label{eq:partition}
   C_2^n(\textbf{i})
   &:=2\textbf{i}+C_2^n \nonumber \\
   &=\{(x_1,x_2,\dots,x_t)\in C_t^n:x_{2j},x_{2j+1}\in\{2i_j,2i_j+1\}, 0\leq j<t/2\}.
   \end{align}
   Note that this generates a partition of $C_t^n$ into $\left(\frac{t}{2}\right)^n$ copies of $C_2^n$. More precisely, $S$ is partitioned into the $\left(\frac{t}{2}-1\right)^n$ copies of $C_2^n$ in (\ref{eq:partition}) such  that $i_j\neq \frac{t-2}{4}$ for all $0\leq j<t/2$; and $C_t^n\setminus S$ is partitioned into the $\left(\frac{t}{2}\right)^n-\left(\frac{t}{2}-1\right)^n$ copies of $C_2^n$ in (\ref{eq:partition}) such that $i_j=\frac{t-2}{4}$ for at least one $0\leq j<t/2$.

    We identify $S$ with a copy of $C_{t-2}^n$ associating each point $(x_1,x_2,\dots,x_t)\in S$ with the point $(x'_1,x'_2,\dots,x'_t)\in C_{t-2}^n$ such that $x'_j=x_j$ if $j<t/2$ and  $x'_j=x_j-2$ if $j\geq t/2$. Under this correspondence, as long as $l_j\neq \frac{t}{2}-1$ or $\frac{t}{2}$, the line $L=\langle l_1,l_2,\dots,l_n \rangle \in \mathcal{L}(C_t^n)$ can be identified with the line $\langle l'_1,l'_2,\dots,l'_n \rangle\in \mathcal{L}(C_{t-2}^n)$, where $l'_j=l_j-2$ if $l_j>\frac{t}{2}$  and $l'_j=l_j$ otherwise.  Note that if $l_j= \frac{t}{2}-1$ or $\frac{t}{2}$ for some $0\leq j<t$, then $L$ is contained in $C_t^n\setminus S$. In fact, if $\textbf{x}_0,\textbf{x}_1,\dots,\textbf{x}_{t-1}$ are the points of $L$ ordered as in Section~\ref{sec:n-cube_definition}, then $\textbf{x}_{2k}$ and $\textbf{x}_{2k+1}$ are in $C_2^n(i_1,i_2,\dots,i_n)$, where $i_j=k$ if $l_j=b$,  $i_j=\frac{t}{2}-1-k$ if $l_j=c$, and $i_j=\left\lfloor\frac{l_j}{2}\right\rfloor$ otherwise.
\begin{figure}[ht!]
    \centering
    \includegraphics[width=1\linewidth]{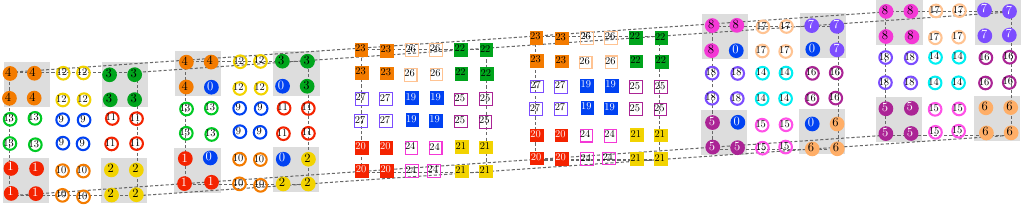}
    \caption{An illustration of the coloring in Theorem~\ref{th:lowernew2} for $t=6$ and $n=3$. The shaded region is a copy of the construction in Figure \ref{fig:cubes_center}, nine colors are used (0-8). The nineteen remaning color-classes are copies of $C_2^3$.}
    \label{fig:t6_n3_28}
\end{figure}
    
    Use the coloring provided by Lemma ~\ref{thm:lowernew} and Lemma~\ref{prop:t=4} for the set $S$ (see the shaded region in Figure \ref{fig:t6_n3_28}).  
      This partial coloring of $C_t^n$ uses $\left(\frac{t-2}{2}\right)^n+\left(\frac{t-2}{4}\right)^n$ colors, each $2^n-1$ or $2^n$ times, and guarantees that all the lines $\langle l_1,l_2,\dots,l_n \rangle$ in $C_t^n$ with $l_j\neq \frac{t}{2}-1$ or $\frac{t}{2}$ for all $0\leq j<t$  are not rainbow.
     Only the lines contained in $C_t^n\setminus S$ are not yet blocked. Since each of these lines contains at least two points in the same copy of $C_2^n$ in the partition of $C_t^n\setminus S$, coloring the $2^n$ points of each of these $\left(\frac{t}{2}\right)^n-\left(\frac{t}{2}-1\right)^n$ copies of $C_2^n$ with a different color ensures that the remaining lines are not rainbow. This completes a balanced rainbow-free coloring of $C_t^n$ that uses a total of 
  $\left(\frac{t-2}{2}\right)^n
  +\left(\frac{t-2}{4}\right)^n
  +\left(\frac{t}{2}\right)^n
  -\left(\frac{t}{2}-1\right)
  =\left(\frac{t}{2}\right)^n+\left\lfloor\frac{t}{4}\right\rfloor^n$ colors, each appearing $2^n-1$ or $2^n$ times. 
\end{proof}

\begin{figure}
    \centering
    \includegraphics[width=0.8\linewidth]{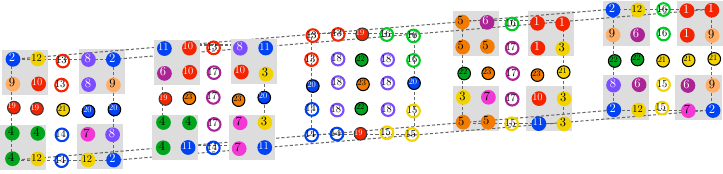}
    \caption{A possible approach to extend Theorem~\ref{thm:mainresult} to odd $t$. This particular example extends the rainbow-free coloring of $C_4^3$ in Figure \ref{fig:small_cases2}c (shaded region) to one of $C_5^3$ with 23 colors.}
    \label{fig:t5_n3_26}
\end{figure}

Note that this lower bound is not significant for $t\geq 4n-2$ due to Theorem~\ref{thm:mainresult}. However, it provides the best known bound for the remaining even values of  $t$. The other classes of $t \pmod 4$ need a dimension-specific analysis to extend the coloring of $C_{t-1}^n$ to $C_t^n$. We illustrate this approach in Figure \ref{fig:t5_n3_26} for  $t=5$ and $n=3$.

\section{Conclusions and future work}\label{sec:conclusions}
So far, we have settled the exact value of the balanced upper chromatic number of $C_t^n$, for any positive integers $n$ and $t \geq 4n-2$ (Theorem~\ref{thm:mainresult}). That is, we determined the largest integer $k$ for which there is a balanced $k$-coloring of $C^n_t$ without rainbow geometric lines. Namely,
\[
\overline{\chi}_b(C_t^n) =\frac{3t^n-(t+2)^n}{2}, \text{ for $t\geq 4n-2$}.
\]

This value still serves as an upper bound for $\overline{\chi}_b(C_t^n)$ for $2/(\sqrt[n]{2}-1)\leq t <4n-2$, but the construction in Section \ref{sec:mainresult} does not work in this range. In this case, Corollary \ref{cor:upper} and Theorem \ref{th:lowernew2} give the following bounds (noticing that the lower bound only works for even $t$),
\[\left(\frac{t}{2}\right)^n+\left\lfloor\frac{t}{4}\right\rfloor^n\leq\overline{\chi}_b(C_t^n)\leq \frac{3t^n-(t+2)^n}{2}, \text{ for } 2/(\sqrt[n]{2}-1)\leq t <4n-2.\]
We conjecture that it is still possible to find a double matching in this range (a balanced coloring would still consist of colors appearing once or twice) and so the upper bound would still be the correct value of $\overline{\chi}_b(C_t^n)$. This would extend Theorem~\ref{thm:mainresult} as follows.

\begin{figure}[hbt]
    \centering
    \includegraphics[width=1\linewidth]{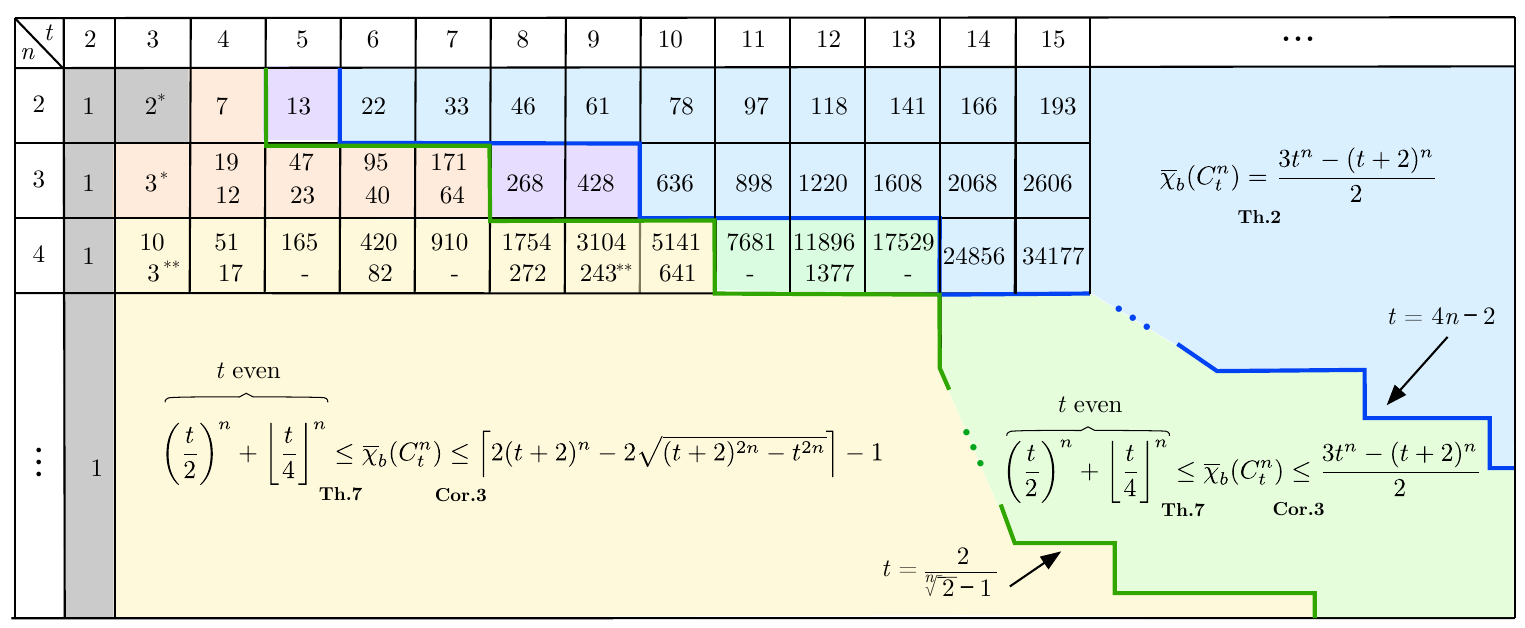}
    \caption{Summary of results: The exact value of $\overline{\chi}_b(C_t^n)$ is known for pairs $(n,t)$ in the grey zone (left, \cite{montejano}) and in the blue zone (upper right, Theorem~\ref{thm:mainresult}); we conjecture that Theorem~\ref{thm:mainresult} can be extended to the green zone (lower right, Conjecture \ref{conj}); we prove this conjecture for $n=2,3$ (purple region); except for a few more improvements, the remaining pairs $(n,t)$ are bounded by Theorem \ref{th:lowernew2} and the second range in Corollary \ref{cor:upper} as shown in the yellow zone (bottom). (*) This is the only case where a smaller upper bound than the one in Corollary \ref{cor:upper} has been found. (**) A different approach than that in Theorem \ref{th:lowernew2} was used for these lower bounds (see comments after Problem 4). The remaining lower bounds for $n=2,3$ are covered by special constructions (orange region).}
    \label{fig:summary}
\end{figure}

\begin{conj}\label{conj}
    For integers $n\geq 2$ and $t\geq \frac{2}{\sqrt[n]{2}-1}$, we have
$$\overline{\chi}_b(C_t^n)=\frac{3t^n-(t+2)^n}{2}.$$
\end{conj}

For the remaining values $2\leq t< 2/(\sqrt[n]{2}-1) $, the second part of Corollary~\ref{cor:upper} and Theorem \ref{th:lowernew2} give the bounds (again we only proved the lower bound for even $t$),

\begin{multline*}
    \left(\frac{t}{2}\right)^n+\left\lfloor\frac{t}{4}\right\rfloor^n
    \leq\overline{\chi}_b(C_t^n)\leq\left\lceil 2(t+2)^n-2\sqrt{(t+2)^{2n}-t^{2n}}\right\rceil-1, \\ \text{ for } 2\leq t< 2/(\sqrt[n]{2}-1).
\end{multline*}

This inequality chain is only tight for $t=2$, giving $\overline{\chi}_b(C_2^n)=1$ for any $n\geq 2$ (as observed earlier). Figure~\ref{fig:summary} summarizes these results. It also includes several improvements to complete the 2-dimensional case and close the gap in dimension 3.

\begin{figure}[ht!]
\centering
\begin{subfigure}[b]{0.25\linewidth}
\includegraphics[width=\linewidth]{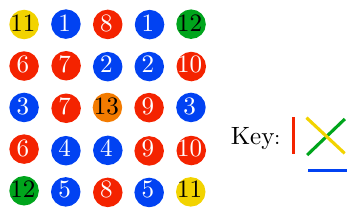}
\caption{A balanced $13$-coloring $C_{5}^2$ with no rainbow lines. The labels $\{1,2,\dots,13\}$ to indicate the colors, and the corresponding double matching is indicated in red, blue, green, and yellow as shown by the key.  }
\label{fig:dim3symm_5}
\end{subfigure}
\hspace{.3in}
\begin{subfigure}[b]{0.6\linewidth}
\includegraphics[width=\linewidth]{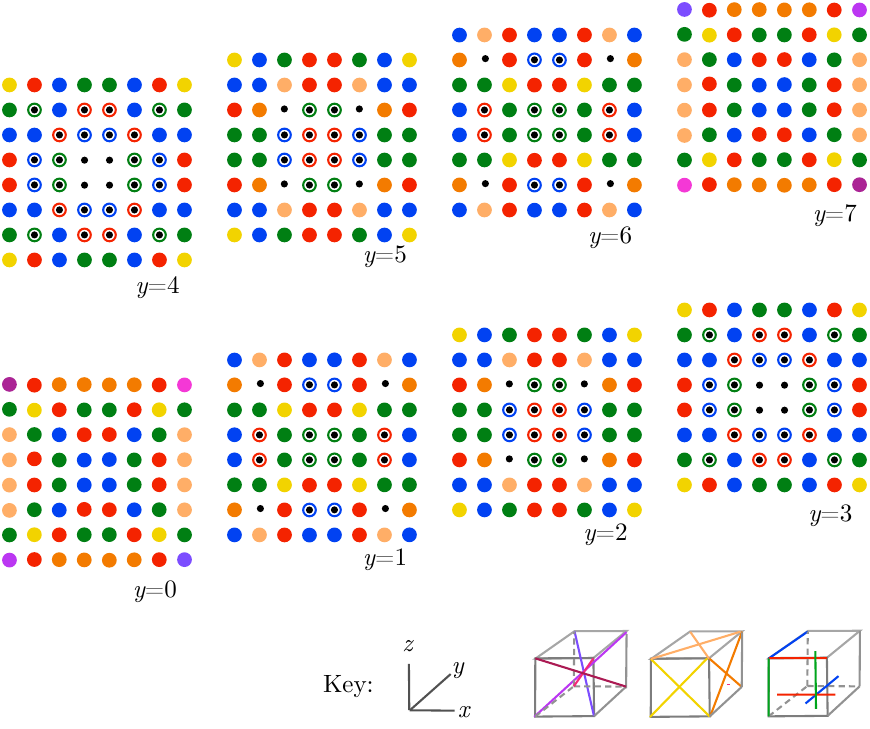}
\caption{A double-matching for $C_{8}^3$. The hollow points are a modification of the general construction for a double-matching of $C_{t}^3$ when $t\geq 10$. In this matching, only $t^3-2|\mathcal{L}(C_8^3)|=t^3-2(3t^2+6t+4)=512-488=24$ points remain unassigned to lines. They are shown by the uncolored vertices.}
\label{dim3symm_8v}
\end{subfigure}
\newline

\begin{subfigure}[c]{0.9\linewidth}
\includegraphics[width=\linewidth]{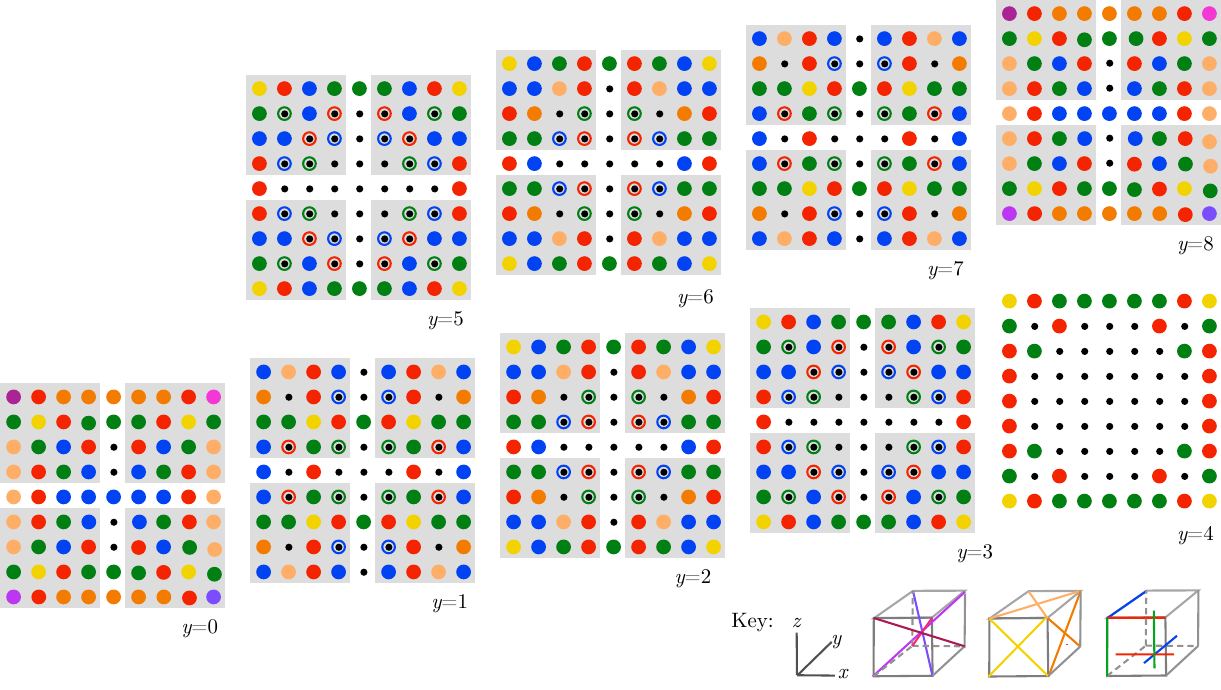}
\caption{A double-matching for $C_{9}^3$; it builts from the matching in part (b) shown by the shaded region.}
\label{dim3symm_9v}
\end{subfigure}
\caption{Double matchings for $C_{5}^2$, $C_{8}^3$, and $C_{9}^3$.}
\label{fig:dim3symm_8}
\end{figure}

Concretely, we prove Conjecture~\ref{conj} for $n=2$ and $n=3$
by providing constructions achieving the upper bound for $n=2$ and $t=5$ in Figure~\ref{fig:dim3symm_5}; for $n=3$ and $t=8$ in Figure~\ref{dim3symm_8v}; and for $n=3$ and $t=9$ in Figure~\ref{dim3symm_9v}. All these constructions are slight modifications of our construction for $t\geq 4n-2$, each is associated to a double-matching as in Section \ref{sec:mainresult}. To complete the 2-dimensional case, the construction in Figure \ref{fig:small_cases2}(a) matches the upper bound in Corollary~\ref{cor:upper}. The refinements of the lower bound for the remaining cases in the space are shown in Figure~\ref{fig:t5_n3_26} for $t=5$, and in  Figures~\ref{fig:small_cases2}(b-e) for  $t=3,4,6,7$, respectively. Besides the case $n=2$ and $t=3$, for which it is known that $\overline{\chi}_b(C_3^2)=2$ \cite{montejano}, there is only one case
for which we have improved the upper bound in Corollary~\ref{cor:upper}, namely, for $n=3$ and $t=3$, where we used a computer program to check that any $4$-coloring of $C_3^3$ contains rainbow lines.

\begin{figure}[htb]
    \centering
    \includegraphics[width=.93\textwidth]{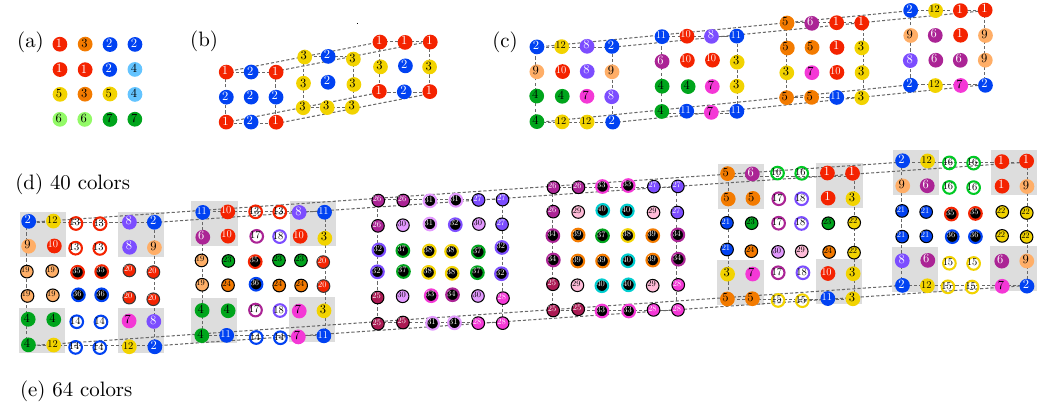}\vspace{-.3in}
    \includegraphics[width=.93\textwidth]{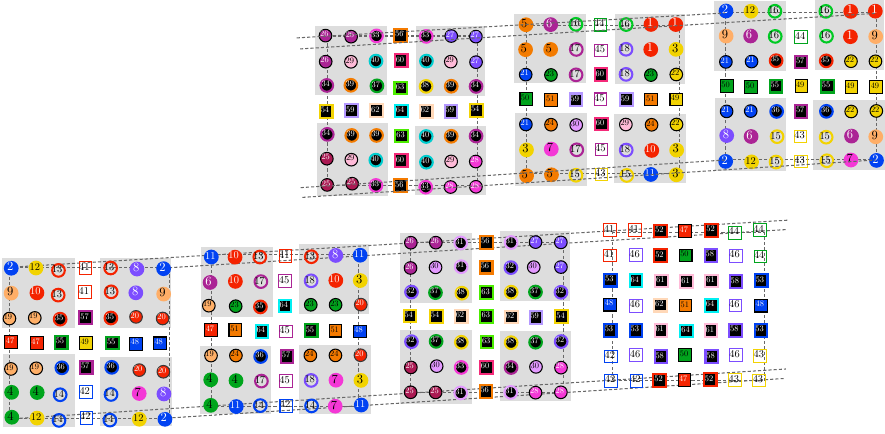}
    \caption{Rainbow-free colorings of (a) $C_4^2$, (b) $C_3^3$, (c) $C_4^3$, (d) $C_6^3$, and (e) $C_7^3$. The shaded regions highlight a coloring of a smaller size (e.g. the shaded region in (d) corresponds to the coloring in (c)). The colorings in (d)-(e) expand the one in (c), so (c)-(e) have color classes of sizes 5 and 6.}
    \label{fig:small_cases2}
\end{figure}

We finish this section by proposing the following open problems. First, it would be interesting to see how the four undetermined values in the space behave in relation to the bound in Corollary~\ref{cor:upper}.

\begin{problem}
    Determine the exact values of $\overline{\chi}_b(C_t^3)$ for $t=4,5,6,7$.
\end{problem}

We note that the lower bound of Theorem \ref{th:lowernew2}, only applies to even values of $t$. A trivial lower bound of $t-1$ applies for odd $t$, but we believe that a modification of the bound  in Theorem \ref{th:lowernew2} can be used for odd $t$, as we did in Figure \ref{fig:t5_n3_26} for $n=3$ and $t=5$; and in Figure \ref{fig:small_cases2}(e) for $n=3$ and $t=7$. 

\begin{problem}
    Provide a nontrivial lower bound for $\overline{\chi}_b(C_t^n)$ when $t$ is odd, perhaps using the construction of Theorem \ref{th:lowernew2}. 
\end{problem}

Lemma \ref{thm:lowernew} gives a lower bound for $\overline{\chi}_b(C_t^n)$ in terms of $\overline{\chi}_b(C_t^d)$, where $d$ is a divisor of $n$. This is in large part what allowed us to prove Theorem \ref{th:lowernew2}. However, in order to use this recursive inequality when $t=p^m$ for some prime $p$, we would need a good lower bound for  $\overline{\chi}_b(C_p^n)$.

\begin{problem}
    Provide a nontrivial lower bound for $\overline{\chi}_b(C_p^n)$ when $p$ is an odd prime. 
\end{problem}

We finish with the following  variation of our parameter that requires the use of all colors exactly the same number of times.
\begin{problem}
    Detemine the maximum number of colors $\overline{\chi}^*_b(C_t^n)$ that a coloring of the points of the  $n$-cubes $C_t^n$ can have without rainbow lines if each color must be used exactly the same number of times. 
\end{problem}
As mentioned in the introduction, these colorings are typically called \emph{equinumerous colorings} \cite{jlmnr03}. Note that $1\leq\overline{\chi}^*_b(C_t^n)\leq\overline{\chi}_b(C_t^n)$ and if $\overline{\chi}_b(C_t^n)$ divides $t^n$, then $\overline{\chi}^*_b(C_t^n)=\overline{\chi}_b(C_t^n)$. This is the case for $n=t=3$, where $\overline{\chi}_b(C_3^3)=3$ (in a balanced 3-coloring of $C_3^3$, each color is used 9 times). This is useful for building balanced rainbow-free colorings in higher dimensions. For example, $C_3^4$ can be partitioned into 3 copies of $C_3^3$ by fixing the last coordinate. Using a copy of the strictly balanced coloring of $C_3^3$ in Figure \ref{fig:small_cases2}(b) for each part, generates a strictly balanced rainbow-free 3-coloring of $C_3^4$. Thus $\overline{\chi}_b(C_3^4)\geq \overline{\chi}^*_b(C_3^4)\geq 3$. Moreover, by Lemma  \ref{thm:lowernew}, $\overline{\chi}_b(C_9^4)\geq (9/3)^4\cdot\overline{\chi}_b(C_3^4)\geq 3^5=243$. In general, $\overline{\chi}_b(C_t^{n+1})\geq \overline{\chi}^*_b(C_t^n)$ and, together with Lemma \ref{thm:lowernew}, 
\[
\overline{\chi}_b(C_{t}^{n+1})\geq \left(\frac{t}{d}\right)^n\overline{\chi}_b(C_d^{n+1})\geq \left(\frac{t}{d}\right)^n\overline{\chi}^*_b(C_d^n),
\]
whenever $d$ divides $t$.

\bibliographystyle{abbrv}
\bibliography{biblio}
\end{document}